\documentclass[12pt,final]{amsart}
\usepackage{amsmath,amssymb}

\usepackage{cite}

\usepackage{color}
\usepackage{soul,graphicx}

\usepackage[color]{showkeys}
\definecolor{refkey}{rgb}{0,0,1}
\definecolor{labelkey}{rgb}{1,0,0}

\newcommand{\eq} [1] {\begin{equation}\label{#1}\quad}
\newcommand{\en} {\end{equation}}

\newtheorem{theorem}{\bf  Theorem}
\newtheorem{lemma}{\bf  Lemma}

\newtheorem{corollary}{\bf \sc Corollary}

\begin{document}
\begin{center}
{\Large \bf Quantitative results on continuity of the spectral factorization mapping in the scalar case}\\[10mm]Lasha Ephremidze$^{1,3}$,  Eugene Shargorodsky$^2$, and  Ilya Spitkovsky$^{1,4}$\\[5mm]
$^1$ Division of Science and Mathematics, New York University Abu Dhabi,  UAE\\
$^2$ Department of Mathematics, King’s College London,  UK\\
$^3$ Razmadze Mathematical Institute of Tbilisi State University, Georgia\\
$^4$ Department of Mathematics, College of William and Mary, Williamsburg, VA \\[8mm]
{\bf Abstract}
\end{center}
\vskip+0.5cm
In the scalar case, the spectral factorization mapping $f\to f^+$ puts a nonnegative integrable function $f$ having an integrable
 logarithm in correspondence with an outer analytic function $f^+$ such that $f = |f^+|^2$ almost everywhere. The main question addressed here is to what extent $\|f^+ - g^+\|_{H_2}$ is controlled by $\|f-g\|_{L_1}$ and $\|\log f - \log g\|_{L_1}$.

 \vskip+0.5cm

 \section{Introduction}
\label{intro}

Let $f$ be a nonnegative integrable function on the unit circle in the complex plain, $0\leq f\in L_1(\mathbb{T})$, satisfying the Paley-Wiener condition
\begin{equation}\label{le1}
\log f\in L_1(\mathbb{T}).
\end{equation}
Then it admits a spectral factorization
\begin{equation}\label{le2}
f(t)=f^+(t)f^-(t)\;\;\text{ a.e. on }\mathbb{T},
\end{equation}
where $f^+$ is a function analytic inside the unit circle, $f^+\in\mathcal{A}(\mathbb{T}_+)$, and $f^-(z)=\overline{f^+(1/\overline{z})}$, which is analytic  outside the unit circle including the infinity, $f^-\in\mathcal{A}(\mathbb{T}_-)$. More specifically, $f^+$ belongs to the Hardy space $H_2(\mathbb{D})$, therefore its boundary values $f^+(t)=f^+(e^{i\theta})=\lim_{r\to 1}f^+(re^{i\theta})$ exist a.e. and the equation \eqref{le2}  holds for these boundary values. Note also that $f^+=\overline{f^-}$ a.e. on $\mathbb{T}$ and therefore \eqref{le2} is equivalent to $$f(t)=|f^+(t)|^2\;\;\text{ a.e. on }\mathbb{T}.$$

Condition \eqref{le1} is necessary for factorization \eqref{le2} to exist. It also plays an important role in the linear prediction theory of stationary stochastic processes, one of the historically first applications of spectral factorization (see \cite{MR0031213}, \cite{MR0009098}). Namely, let $\ldots, X_{-1}, X_0, X_1,\ldots$ be a stationary stochastic process with the spectral measure $d\mu=f\,dt+d\mu_s$. In a different but equivalent language,
$\{X_n\}_{n\in\mathbb{Z}}$ is a sequence in a Hilbert space and $\langle X_n,X_k\rangle=\frac{1}{2\pi}\int_\mathbb{T} e^{i(n-k)\theta}\,d\mu(\theta)$.
The process is {\em deterministic} if $X_{n+1}$ can be represented as the limit of linear combinations of vectors from $\{\ldots, X_{n-1}, X_n\}$, i.e, $X_{n+1}\in \overline{\mbox{Span}}\{\ldots, X_{n-1}, X_n\}$. As it happens (see e.g. \cite{MR0009098}), condition  \eqref{le1} is necessary and sufficient for the process to be non-deterministic.

Starting with original applications in the prediction theory of stochastic processes, spectral factorization procedure appeared in such seemingly distant areas as singular integral equations \cite{MR0102720}, \cite{MR2663312}, linear estimation \cite{Kai99}, quadratic and $H_\infty$ control \cite{MR0335000}, \cite{MR932459}, \cite{MR2663312}, communications \cite{fisher},  filter design \cite{MR1162107}, \cite{MR1712468}, \cite{MR1411910},  etc.

If we require $f^+$ to be an outer analytic function, then the factorization \eqref{le2} is unique up to a constant factor $c$ with absolute value 1, $|c|=1$. The unique spectral factor which is positive at the origin can be a priori written as
\begin{equation}\label{3}
 f^+(z)=\exp\left(\frac 1{4\pi}
\int\nolimits_0^{2\pi}\frac{e^{i\theta}+z}{e^{i\theta}-z}\log
f(e^{i\theta})\,d\theta\right).
\end{equation}
In most applications, a spectral factor $f^+$ in \eqref{le2} is not explicitly required to be outer and instead is subject to  certain extremal conditions called, in various works, minimal phase or maximal energy, optimal, etc. In mathematical terms, however, they amount to $f^+$ being outer, so seeking the solution \eqref{3} is natural.

From the practical point of view, it is important to study the continuity properties of the spectral factorization map
\begin{equation}\label{4}
f\mapsto f^+
\end{equation}
defined by \eqref{3}. Namely, we are interested in knowing how close  $g^+$ is to $f^+$ when a spectral density $g$ is close to $f$. The reason why we study this question is that usually an estimated spectral density function $\hat{f}$ being dealt with is constructed empirically  from observations and is only an approximation to the theoretically existing spectral density $f$. Therefore we need to know how close $\hat{f}^+$ remains to $f^+$ under such approximation.

An answer to the above question depends on norms we use as a measurement of the accuracy in the spaces of functions and of their spectral factors. Since the boundary values of the function \eqref{3} can be expressed as
\begin{equation*}\label{5}
f^+(t)=\sqrt{f(t)}\exp\left(\frac{i}{2}\widetilde{\log f}(t)\right),
\end{equation*}
where ${}^\sim$ stands for the harmonic conjugation operator
$$
\tilde{h}(e^{i\tau})=(P)\frac{1}{2\pi}\int_0^{2\pi}h(e^{i\theta})\cot(\frac{\tau-\theta}{2})\,d\theta,
$$
and the conjugation is not a bounded operator on $L_\infty$ or $C(\mathbb{T})$, it is not surprising that the map \eqref{4} is not continuous in these spaces \cite{MR827391}. Furthermore, it is shown in  \cite{Boche} that every continuous function on $\mathbb{T}$ is a discontinuity point of the spectral factorization mapping in the uniform norm, whereas in \cite{MR1855307} it was shown that on a large class of function spaces (the so called decomposing Banach algebras) the spectral factorization mapping is continuous.

The spectral factorization of a trigonometric polynomial
\begin{equation}\label{7}
f(t)=\sum_{k=-N}^N c_k t^k,
\end{equation}
which is non-negative on $\mathbb{T}$, has the form
\begin{equation*}\label{8}
f(t)=\sum_{k=0}^N a_k t^k \sum_{k=0}^N \overline{a_k} t^{-k},
\end{equation*}
i.e. the spectral factor $f^+$ is a polynomial of the same degree $N$. This result is known as  the Fej\'{e}r-Riesz lemma (see, e.g. \cite{MR1162107}). The spectral factor can also be expressed in terms of zeros of polynomial \eqref{7}, and therefore the map \eqref{4} is continuous on $\mathcal{P}_N$, the set of all functions of the form \eqref{7}. Papers \cite{MR2446566}, \cite{Boche2} are devoted to estimating the constant $C_N$ in the inequality
$$
\|\phi^+-\psi^+\|_{L_\infty}\leq C_N\|\phi-\psi\|_{L_\infty},\;\;\;\phi, \psi\in\mathcal{P}_N,
$$
and it is shown there that $C_N\sim\log N$ asymptotically, under the condition that  the values of functions $\phi$ and $\psi$ are bounded away from $0$.

Moving to Lebesgue spaces,  the map \eqref{4} is not continuous in the $L_1$ norm in general, since a small change of values of function $f$, if these values are close to $0$, may cause a significant change of $\log f$. However,
\begin{equation}\label{9}
\|f_n-f\|_{L_1}\to 0 \text{ and }\|\log f_n-\log f\|_{L_1}\to 0 \Longrightarrow \|f^+_n-f^+\|_{H_2}\to 0.
\end{equation}
A proof of an analogue of \eqref{9} for more general matrix case can be found in \cite{MR2096876} or \cite{MR2838115}. In the present paper, we discuss  quantitative estimates of the rate in the above convergence. Firstly, we look for  estimates of $\|g^+-f^+\|_{H_2}$ in terms of $\|g-f\|_{L_1}$ and
$\|\log g-\log f\|_{L_1}$. It turns out that, in general, there is no such  estimate. Namely, there is no function
$\Pi : [0, +\infty)^2 \to [0, +\infty)$ such that
$\lim_{s, t \to 0} \Pi(s, t) = 0$ for which the estimate
$$
\|g^+ - f^+\|_{H_2}^2 \leq\Pi\left(\|g - f\|_{L_1}, \|\log g - \log f\|_{L_1}\right)
$$
holds for all $f, g \geq0$ with $\|f\|_{L_1}, \|g\|_{L_1} \leq 1$.
\begin{theorem}\label{No1}
There exist functions $f_n, g_n \ge 0$, $n \in \mathbb{N}$, such that
$$
\|f_n\|_{L_1}, \|g_n\|_{L_1} \le 1, \ \|g_n - f_n\|_{L_1} \le \frac1n\, , \ \|\log g_n - \log f_n\|_{L_1} \le \frac1n\, ,
$$
but $\|g_n^+ - f_n^+\|_{H_2} \ge 2 - 1/n$.
\end{theorem}

Nevertheless one can still obtain an estimate for $\|g^+ - f^+\|_{H_2}$ if one takes into account that for each
$f \in L_1(\mathbb{T})$  there exists an Orlicz space $L_\Psi(\mathbb{T})$ such that $f \in L_\Psi(\mathbb{T})$
(see, e.g., \cite[$\S$8]{orlicz2}).
One can show that there exists a function $\Pi_\Psi : [0, +\infty) \to [0, +\infty)$ such that
$\lim_{t \to 0} \Pi_\Psi(t) = 0$ and
\begin{equation}\label{0rlicz}
\|f^+ - g^+\|_{H_2}^2 \le 2\|f - g\|_{L_1} +  \|f\|_{\Psi}\, \Pi_\Psi\left(\|\log f - \log g\|_{L_1}\right)
\end{equation}
(see Theorem \ref{main} below).
The estimate becomes particularly simple if $f \in L_\infty(\mathbb{T})$.

\begin{theorem}\label{No2}
Let $f$ and $g$ be arbitrary spectral densities for which $f^+$ and $g^+$ exist. Then
\begin{equation*}\label{10}
\|f^+ - g^+\|_{H_2}^2 \le 2\|f - g\|_{L_1} + 2.5\, \|f\|_{L_\infty} \|\log f - \log g\|_{L_1}.
\end{equation*}
\end{theorem}

The paper is organized as follows. In Sect.~\ref{sec:1}, we prove \eqref{0rlicz} and Theorem \ref{No2}. Theorem \ref{No1} is proved in Sect.~\ref{sec:2}.

This paper is  a preliminary step towards the investigation of  similar problems in the more complicated matrix case, which is going to be the subject of a forthcoming paper.

\section{Positive Results}\label{sec:1}
Let $K$ be the best constant in Kolmogorov's weak type $(1, 1)$ inequality
$$
m \{\vartheta \in [-\pi, \pi) : \ |\widetilde{\psi}(\vartheta)| \ge \lambda\} \leq \frac{K}{\lambda}\int_{-\pi}^\pi|\psi(\vartheta)|\,d\vartheta,\;\;\lambda>0,\;\;\psi\in L_1[-\pi,\pi),
$$
where $m$ stands for the Lebesgue measure on the real line. It is known that $K=(1+3^{-2}+5^{-2}+\ldots)/(1-3^{-2}+5^{-2}-\ldots)\approx  1.347$
(see \cite{Davis1}).

\begin{lemma}\label{1-1}
Let $G : [0, +\infty) \to [0, +\infty)$ be a bounded absolutely continuous function that attains its maximum at point
$a \in (0, +\infty)$ $($and possibly elsewhere$)$ and is nondecreasing on $[0, a]$. Suppose $G(0) = 0$ and
$$
I(G) := \int_0^a G'(\lambda)\, \frac{d\lambda}{\lambda} < +\infty .
$$
Then
$$
\int_{-\pi}^\pi
G\left(\left|\widetilde{\psi}(\vartheta)\right|\right) d\vartheta \le K I(G) \|\psi\|_{L_1}.
$$
\end{lemma}
\begin{proof}
Let
$$
\mu_\psi(\lambda) := \left|\left\{\vartheta \in [-\pi, \pi] : \ \left|\widetilde{\psi}(\vartheta)\right| \ge \lambda\right\}\right| .
$$
Using Kolmogorov's weak type (1, 1) estimate with constant $K$, one gets
\begin{eqnarray*}
&& \int_{-\pi}^\pi G\left(\left|\widetilde{\psi}(\vartheta)\right|\right) d\vartheta \le
\int_{\left|\widetilde{\psi}(\vartheta)\right| < a} G\left(\left|\widetilde{\psi}(\vartheta)\right|\right)\, d\vartheta +
G(a) \mu_\psi(a) \\
&& = \int_{\left|\widetilde{\psi}(\vartheta)\right| < a} \int_0^{\left|\widetilde{\psi}(\vartheta)\right|} G'(\lambda)\, d\lambda\, d\vartheta +
G(a) \mu_\psi(a) \\
&& = \int_0^a G'(\lambda) (\mu_\psi(\lambda) - \mu_\psi(a))\, d\lambda + G(a) \mu_\psi(a) \\
&& =  \int_0^a G'(\lambda) \mu_\psi(\lambda)\, d\lambda - G(a) \mu_\psi(a) + G(a) \mu_\psi(a)  \\
&& = \int_0^a G'(\lambda) \mu_\psi(\lambda)\, d\lambda
\le K \|\psi\|_{L_1} \int_0^a G'(\lambda)\, \frac{d\lambda}{\lambda} = K I(G) \|\psi\|_{L_1}\, .
\end{eqnarray*}
\end{proof}

We need some notation from the theory of Orlicz  spaces (see \cite{orlicz2}, \cite{orlicz3}).
 Let $\Phi$ and
$\Psi$ be mutually complementary $N$-functions, i.e.
$$
\Phi(x)=\int_0^{|x|}u(t)\,dt \;\;\;\text{ and }\;\;\; \Psi(x)=\int_0^{|x|}v(t)\,dt,
$$
where $u:[0,\infty)\longrightarrow [0,\infty)$ is a right continuous, nondecreasing function with
$u(0)=0$ and $u(\infty):=\lim_{t\to\infty}u(t)=\infty$, and $v$ is defined by the equality $v(x)=\sup_{u(t)\leq x}t$.
Let $(\Omega, \Sigma, \mu)$ be a measure space, and let $L_\Phi(\Omega)$,
$L_\Psi(\Omega)$ be the corresponding Orlicz spaces, i.e.  $L_\Phi(\Omega)$ is the set of measurable functions on $\Omega$ for which either of the following  norms
$$
\|f\|_{\Phi} = \sup\left\{\left|\int_\Omega f g d\mu\right| : \
\int_\Omega \Psi(g) d\mu \le 1\right\}
$$
and
$$
\|f\|_{(\Phi)} = \inf\left\{\kappa > 0 : \
\int_\Omega \Phi\left(\frac{f}{\kappa}\right) d\mu \le 1\right\}
$$
is finite.  Note that these two norms are equivalent, namely  (see, e.g., \cite[(9.24)]{orlicz2} or \cite[\S 3.3, (14)]{orlicz3})
\begin{equation*}\label{Luxemburgequiv}
\|f\|_{(\Phi)} \le \|f\|_{\Phi} \le 2 \|f\|_{(\Phi)}\, , \ \ \ \forall f \in L_\Phi(\Omega) .
\end{equation*}
We will use the following H\"older inequality (see, e.g., \cite[(9.27)]{orlicz2} or \cite[\S 3.3, (16)]{orlicz3})
\begin{equation}\label{Hoel}
\left|\int_\Omega f g d\mu\right| \le \|f\|_{\Psi} \|g\|_{(\Phi)} .
\end{equation}

For an $N$-function  $\Phi$, let

\begin{equation}\label{Inv}
\Lambda_\Phi(s) := \inf\left\{t > 0 : \ \frac1t\, \Phi'\left(\frac1t\right) \le \frac1s\right\} , \ \ \ s > 0.
\end{equation}
If $\Phi'$ is continuous, the above definition of
$\Lambda_\Phi$ can be rewritten in terms of inverse functions, because
$\Phi'$ is nondecreasing.  For an arbitrary $N$-function $\Phi$, one has
$$
\tau \Phi'(\tau) \le \int_\tau^{2\tau} \Phi'(x)\, dx = \Phi(2\tau) - \Phi(\tau) < \Phi(2\tau) .
$$
Hence
\begin{equation*}
\Lambda_\Phi(s) \le  \frac2{\Phi^{-1}\left(\frac1s\right)}\, , \ \ \ s > 0.
\end{equation*}
It is clear that
\begin{equation}\label{sh17}
\Lambda_\Phi(s) \to 0 \ \mbox{ as } s \to 0+.
\end{equation}
Also,
\begin{equation}\label{p}
\Phi(\tau) \equiv \tau^q/q , \ 1 < q < \infty  \ \ \ \Longrightarrow \ \ \ \Lambda_\Phi(s) \equiv s^{1/q} .
\end{equation}

\begin{lemma}\label{Orl}
For every $N$-function $\Phi$, the following estimate holds
\begin{equation}\label{js1}
\left\|1 - \cos \widetilde{\psi}\right\|_{(\Phi)} \le 2 \Lambda_\Phi\left(K_0\|\psi\|_{L_1}\right)\, , \ \ \ \forall \psi \in L_1 ,
\end{equation}
where
\begin{equation}\label{K0}
K_0 := \frac{K}{2}\, \int_0^\pi  \frac{\sin \lambda}\lambda\, d\lambda < 1.25
\end{equation}
and $K$ is the same as in Lemma $\ref{1-1}$.
\end{lemma}
\begin{proof}
We will use Lemma \ref{1-1} with
$$
G(\lambda) = \Phi\left(\frac{1 - \cos\lambda}{\kappa}\right)
$$
and $a = \pi$. We have
\begin{eqnarray*}
I(G) = \frac1\kappa \int_0^\pi \Phi'\left(\frac{1 - \cos\lambda}{\kappa}\right) \sin \lambda\, \frac{d\lambda}{\lambda} \le
\frac1\kappa\, \Phi'\left(\frac{2}{\kappa}\right) \int_0^\pi  \frac{\sin \lambda}\lambda\, d\lambda\, .
\end{eqnarray*}
Hence
\begin{equation}\label{20+}
\int_{-\pi}^\pi \Phi\left(\frac{1 - \cos\widetilde{\psi}(\vartheta)}{\kappa}\right) d\vartheta \le  \frac2\kappa\, \Phi'\left(\frac{2}{\kappa}\right)
\frac{K}{2}\, \int_0^\pi  \frac{\sin \lambda}\lambda\, d\lambda\, \|\psi\|_{L_1}\, .
\end{equation}
Taking $\kappa>2 \Lambda_\Phi\left(K_0\|\psi\|_{L_1}\right)$, we observe that the right-hand side of inequality \eqref{20+} does not exceed 1 by virtue of the definition \eqref{Inv}. Thus \eqref{js1} follows.
\end{proof}

\begin{lemma}\label{L1}
\begin{equation}\label{1L}
\left\|1 - \cos \widetilde{\psi}\right\|_{L_1} \le 2 K_0\|\psi\|_{L_1} , \ \ \ \forall \psi \in L_1 ,
\end{equation}
where $K_0$ is defined by \eqref{K0}.
\end{lemma}
\begin{proof}
We will use Lemma \ref{1-1} with
$$
G(\lambda) = 1 - \cos\lambda
$$
and $a = \pi$. We have
\begin{eqnarray*}
I(G) =  \int_0^\pi  \frac{\sin \lambda}\lambda\, d\lambda
\end{eqnarray*}
and
$$
\left\|1 - \cos \widetilde{\psi}\right\|_{L_1} =
\int_{-\pi}^\pi \left(1 - \cos\widetilde{\psi}(\vartheta)\right) d\vartheta \le
K \int_0^\pi  \frac{\sin \lambda}\lambda\, d\lambda\, \|\psi\|_{L_1}\, ,
$$
which implies \eqref{1L}.
\end{proof}

\begin{theorem}\label{main}
For every pair $\Phi$ and
$\Psi$ of mutually complementary $N$-functions, the following estimate holds
\begin{equation*}\label{js2}
\|f^+ - g^+\|_{H_2}^2 \le 2\|f - g\|_{L_1} + 4 \|f\|_{\Psi}\, \Lambda_\Phi\left(\frac{K_0}2\, \|\log f - \log g\|_{L_1}\right)\, ,
\end{equation*}
where $K_0$ is the same as in Lemma $\ref{Orl}$.
\end{theorem}
\begin{proof}
\begin{eqnarray*}
&& \|f^+ - g^+\|_{H_2}^2 = \|f^+\|_{H_2}^2 + \|g^+\|_{H_2}^2 \\
&& - 2\, \mathrm{Re} \int_{-\pi}^\pi f^{1/2}(\vartheta) g^{1/2}(\vartheta)
\exp\left(\frac{i}{2}\, (\log f(\vartheta) - \log g(\vartheta))^\sim \right)\, d\vartheta\\
&& = \|f^{1/2}\|_{L_2}^2 + \|g^{1/2}\|_{L_2}^2  -2 \int_{-\pi}^\pi f^{1/2}(\vartheta) g^{1/2}(\vartheta)\, d\vartheta\\
&& +  2  \int_{-\pi}^\pi f^{1/2}(\vartheta) g^{1/2}(\vartheta)
\left(1 - \cos\Big(\frac{1}{2}\, (\log f(\vartheta) - \log g(\vartheta))^\sim \Big)\right)\, d\vartheta \\
&& = \left\|(f^{1/2} - g^{1/2})^2\right\|_{L_1} \\
&& +  2  \int_{-\pi}^\pi f^{1/2}(\vartheta) \left(g^{1/2}(\vartheta) - f^{1/2}(\vartheta)\right)
\left(1 - \cos\Big(\frac{1}{2}\, (\log f(\vartheta) - \log g(\vartheta))^\sim \Big)\right)\, d\vartheta \\
&& +  2  \int_{-\pi}^\pi f(\vartheta)
\left(1 - \cos\Big(\frac{1}{2}\, (\log f(\vartheta) - \log g(\vartheta))^\sim \Big)\right)\, d\vartheta \\
&& \le \int_{-\pi}^\pi \left(g^{1/2}(\vartheta) - f^{1/2}(\vartheta)\right)^2\, d\vartheta \\
&& +  4  \int_{-\pi}^\pi f^{1/2}(\vartheta) \left(g^{1/2}(\vartheta) - f^{1/2}(\vartheta)\right)_+\, d\vartheta \\
&& +  2  \int_{-\pi}^\pi f(\vartheta)
\left(1 - \cos\Big(\frac{1}{2}\, (\log f(\vartheta) - \log g(\vartheta))^\sim \Big)\right)\, d\vartheta .
\end{eqnarray*}
Here and below, for $x\in\mathbb{R}$, we define $x_+:=\max(x,0)$.

Using the H\"older inequality \eqref{Hoel} and the elementary inequality
$$
\left(a^{1/2} - b^{1/2}\right)^2 + 2b^{1/2}\left(a^{1/2} - b^{1/2}\right)_+ \le |a - b| , \ \ \ \forall a, b \ge 0 ,
$$
which is easily proved by considering the cases $a \ge b$ and $a < b$ separately, one gets
$$
\|f^+ - g^+\|_{H_2}^2 \le 2 \|f - g\|_{L_1} + 2\|f\|_{\Psi} \left\|1 - \cos \Big(\frac{1}{2}\, (\log f - \log g)^\sim \Big)\right\|_{(\Phi)}\, .
$$
It is now left to apply Lemma \ref{Orl} with
$
\psi = \frac{1}{2}\, (\log f - \log g)\, .
$
\end{proof}

Since every integrable function belongs to a certain Orlicz  space (see \cite[\S 8]{orlicz2}), Theorem \ref{main} with an appropriate
pair $\Phi$ and $\Psi$ of mutually complementary $N$-functions applies to any
nonnegative integrable function $f$ with an integrable logarithm. The condition \eqref{sh17} is fulfilled as well.

\begin{corollary}
For every $p \in (1, \infty)$, there exists a constant $C(p)$ such that
\begin{equation*}
\|f^+ - g^+\|_{H_2}^2 \le 2 \|f - g\|_{L_1} + C(p) \|f\|_{L_p} \|\log f - \log g\|_{L_1}^{\frac{p - 1}{p}}\, .
\end{equation*}
One can take $C(p) = 2^{\frac{p + 1}{p}} K_0^{\frac{p - 1}{p}}\left(\frac{p}{p-1}\right)^{\frac{p - 1}{p}}$, where $K_0$ is defined by \eqref{K0}.
\end{corollary}
\begin{proof}
It is sufficient to take $\Psi(t) \equiv t^p/p$,  $\Phi(t) \equiv t^q/q$,  $1 < p < \infty$,  $q = p/(p - 1)$ \ in Theorem \ref{main} and to
apply \eqref{p} and \cite[(9.7)]{orlicz2}.
\end{proof}

\begin{corollary}
There exists a constant $C$ such that
\begin{equation*}
\|f^+ - g^+\|_{H_2}^2 \le 2\|f - g\|_{L_1} + C \|f\|_{L_\infty} \|\log f - \log g\|_{L_1}\, .
\end{equation*}
One can take $C = 2 K_0 < 2.5$, where $K_0$ is defined by \eqref{K0}.
\end{corollary}
\begin{proof}
It follows from the proof of Theorem \ref{main} that
$$
\|f^+ - g^+\|_{H_2}^2 \le  2\|f - g\|_{L_1} + 2\|f\|_{L_\infty} \left\|1 - \cos \Big(\frac{1}{2}\, (\log f - \log g)^\sim \Big)\right\|_{L_1}\, .
$$
It is now left to apply Lemma \ref{L1}.
\end{proof}

\section{Negative Results}
\label{sec:2}

In this section we prove Theorem \ref{No1}.
\begin{proof}
It follows from the proof of Theorem \ref{main} that for any $f, g \ge 0$ one has
\begin{eqnarray*}
&& \|f^+ - g^+\|_{H_2}^2 =  \left\|(f^{1/2} - g^{1/2})^2\right\|_{L_1} \\
&& +  2  \int_{-\pi}^\pi f^{1/2}(\vartheta) \left(g^{1/2}(\vartheta) - f^{1/2}(\vartheta)\right)
\left(1 - \cos\Big(\frac{1}{2}\, (\log f(\vartheta) - \log g(\vartheta))^\sim \Big)\right)\, d\vartheta \\
&& +  2  \int_{-\pi}^\pi f(\vartheta)
\left(1 - \cos\Big(\frac{1}{2}\, (\log f(\vartheta) - \log g(\vartheta))^\sim \Big)\right)\, d\vartheta \\
&& \ge  - 4  \int_{-\pi}^\pi f^{1/2}(\vartheta) \left|g^{1/2}(\vartheta) - f^{1/2}(\vartheta)\right|\, d\vartheta \\
&& +  2  \int_{-\pi}^\pi f(\vartheta)
\left(1 - \cos\Big(\frac{1}{2}\, (\log f(\vartheta) - \log g(\vartheta))^\sim \Big)\right)\, d\vartheta \\
&& \ge  2  \int_{-\pi}^\pi f(\vartheta)
\left(1 - \cos\Big(\frac{1}{2}\, (\log f(\vartheta) - \log g(\vartheta))^\sim \Big)\right)\, d\vartheta - 4 \|f - g\|_{L_1}\, .
\end{eqnarray*}

Let $w_n$ be a conformal mapping of
the unit disk onto the ellipse with the axes
$$
[-\varepsilon_n, 0]  \ \  \mbox{ and } \ \  -\varepsilon_n/2 + i[-2\pi, 2\pi]  ,
$$
such that $w_n(0) = -\varepsilon_n/2$, where $\varepsilon_n = \frac1{2\pi n}$.
Let $h_n := \exp(\mathrm{Re}\, w_n)$. Then $(\log h_n)^\sim = (\mathrm{Re}\, w_n)^\sim = \mathrm{Im}\, w_n$ and therefore $\|1 - \cos(\frac{1}{2}\, (\log h_n)^\sim )\|_{L_\infty}=2$.

Due to duality considerations, there exists $f^0_n \ge 0$ such that $\|f^0_n\|_{L_1} = 1$ and
\begin{eqnarray*}
\int_{-\pi}^\pi f^0_n(\vartheta)
\left(1 - \cos\Big(\frac{1}{2}\, (\log h_n(\vartheta))^\sim \Big)\right)\, d\vartheta \\
\ge \left(1 - \frac{\varepsilon_n}2\right)
\left\|1 - \cos\Big(\frac{1}{2}\, (\log h_n)^\sim \Big)\right\|_{L_\infty}\, .
\end{eqnarray*}
If $\log f^0_n \in L_1$, we take $f_n = f^0_n$. Otherwise we define $f_n = (1 - \varepsilon_n/2) f^0_n + \varepsilon_n/4\pi$.
Then $\|f_n\|_{L_1} = 1$, and $(1 - \varepsilon_n/2)^2 > 1 - \varepsilon_n$ implies
\begin{eqnarray*}
\int_{-\pi}^\pi f_n(\vartheta)
\left(1 - \cos\Big(\frac{1}{2}\, (\log h_n(\vartheta))^\sim \Big)\right)\, d\vartheta \\
\ge \left(1 - \varepsilon_n\right)
\left\|1 - \cos\Big(\frac{1}{2}\, (\log h_n)^\sim \Big)\right\|_{L_\infty}\, .
\end{eqnarray*}
Finally, let $g_n = h_n f_n$. Then $0 \le g_n \le f_n$, $\|g_n\|_{L_1} \le 1$,
\begin{eqnarray*}
&& \|f_n - g_n\|_{L_1} = \|f_n(1 - h_n)\|_{L_1} \le \|1 - h_n\|_{L_\infty} \le 1 - e^{-\varepsilon_n} \le \varepsilon_n < \frac1{2n}\, , \\
&& \|\log f_n - \log g_n\|_{L_1} = \|\log h_n\|_{L_1}\le 2\pi \|\log h_n\|_{L_\infty} = 2\pi \varepsilon_n = \frac1n\, ,
\end{eqnarray*}
and
\begin{eqnarray*}
&& \|f_n^+ - g_n^+\|_{H_2}^2  \ge  2  (1 - \varepsilon_n)
\left\|1 - \cos\Big(\frac{1}{2}\, (\log h_n)^\sim \Big)\right\|_{L_\infty} - 4 \|f - g\|_{L_1} \\
&& > 4(1 - \varepsilon_n) - \frac2n > 4 \left(1 - \frac1{4 n}\right) - \frac2n  \ge \left(2 - \frac1n\right)^2 .
\end{eqnarray*}
\end{proof}

{\bf Remark.} \ The norms  $\|\log f_n\|_{L_1}$ and $\|\log g_n\|_{L_1}$ might not be bounded in Theorem \ref{No1}.
Let $f^0_n$ be the function from the above proof. Changing the definition of $f_n$ in the proof to $f_n = f^0_n + 1$,
one can change the estimates $\|f_n\|_{L_1}, \|g_n\|_{L_1} \le 1$ in the theorem for $\|f_n\|_{L_1} = 2\pi + 1$,
$\|g_n\|_{L_1} \le 2\pi + 1$, $\|\log f_n\|_{L_1} \le 1$.


%
%

\def\cprime{$'$}

\end{document}